\documentclass{article}
\usepackage[utf8]{inputenc}
\usepackage[T1]{fontenc}
\usepackage{amssymb,amsmath,amsthm,mathrsfs}
\usepackage[english]{babel}
\usepackage{mathtools}
\usepackage{mathrsfs}
\usepackage{bm}
\usepackage{enumerate}
\usepackage{authblk}
\usepackage{graphicx}
\usepackage{subcaption}
\usepackage{comment}
\usepackage[margin=0.3in]{geometry}
\newgeometry{includefoot,left=3cm,right=3cm, bottom=3cm,top=3cm}
\usepackage{color}

\newcommand{\E}{\mathbb{E}}
\newcommand{\p}{\mathbb{P}}
\newcommand{\UH}{\mathbb{H}}
\newcommand{\C}{\mathbb{C}}
\newcommand{\R}{\mathbb{R}}
\newcommand{\D}{\mathbb{D}}

\newcommand{\Z}{\mathbb{Z}}
\newcommand{\dist}{\mathrm{dist}}
\newcommand{\SLE}{\mathrm{SLE}}
\newcommand{\CLE}{\mathrm{CLE}}

\newcommand{\strip}{\mathscr{S}}

\newcommand{\one}{\mathbf{1}}

\renewcommand{\epsilon}{\varepsilon}

\newcommand{\G}{\mathscr{G}}

\newcommand{\wt}{\widetilde}
\newcommand{\wh}{\widehat}
\newcommand{\cont}{\mathrm{Cont}}

\newcommand{\im}{\mathrm{Im}}
\newcommand{\FH}{\mathfrak{H}}
\newcommand{\Fh}{\mathfrak{h}}
\newcommand{\area}{\mathrm{Area}}
\newcommand{\CM}{\mathcal{M}}
\newcommand{\CW}{\mathcal{W}}

\newcommand{\hcap}{\mathrm{hcap}}
\newcommand{\Rad}{\mathrm{R}}
\newcommand{\Lat}{\mathrm{L}}
\newcommand{\Var}{\mathrm{Var}}

\newcommand{\ol}{\overline}

\newcommand{\argmax}{\arg \max}

\newcommand{\qwedgeA}[2]{\mathsf{QWedge}_{\bm{\gamma}=#1}^{\bm{\alpha}=#2}}

\newtheorem{thm}{Theorem}[section]
\newtheorem{lem}[thm]{Lemma}

\theoremstyle{definition}
\newtheorem{defin}[thm]{Definition}
\newtheorem{rmk}[thm]{Remark}

\begin{document}

\title{A Gaussian free field approach to the natural parametrisation of $\SLE_4$}

\author{Vlad Margarint and Lukas Schoug}

\date{\today}
\maketitle

\parindent 0 pt
\setlength{\parskip}{0.20cm plus1mm minus1mm}


\begin{abstract}
We construct the natural parametrisation of $\SLE_4$ using the Gaussian free field, complementing the corresponding results for $\SLE_\kappa$ for $\kappa \in (0,4)$ by Benoist \cite{benoist2018natural} and for $\kappa \in (4,8)$ by Miller and Schoug \cite{ms2022volume}.
\end{abstract}

\section{Introduction}
The Schramm-Loewner evolution, $\SLE_\kappa$ ($\kappa > 0$), introduced in \cite{schramm2000scaling} is a one-parameter family of random conformally invariant curves which arise as the scaling limits of statistical mechanics models in two dimensions. When constructing $\eta \sim \SLE_\kappa$ via the Loewner differential equation (see Section~\ref{sec:SLE}), the curve is parametrised by capacity, that is, if $K_t$ is the set of points in $\UH$ which are separated from $\infty$ by $\eta([0,t])$, then $\hcap(K_t) = 2t$ for each $t \geq 0$. However, there are other parametrisations which may be equally natural for $\SLE$. One example is the so-called \emph{natural parametrisation} of $\SLE$, constructed in \cite{ls2011natural} for $\kappa < 4(7-\sqrt{33})$ and then in \cite{lz2013natural} for all $\kappa < 8$. This is the parametrisation which is conjectured to arise when considering $\SLE_\kappa$ as a scaling limit of a discrete model where the discrete interface is parametrised by the number of vertices it traverses (so far this has only been proved to be the case when considering the scaling limit of loop-erased random walk \cite{lv2021convergence} and critical site percolation on the triangular lattice \cite{hls2022percolation}). It turns out that the natural parametrisation of $\SLE_\kappa$ is in fact (a deterministic constant times) the $d_\kappa$-dimensional Minkowski content of $\SLE_\kappa$, where $d_\kappa = 1 +\kappa/8$ is the almost sure Hausdorff dimension of $\SLE_\kappa$ \cite{beffara2008dimension}.

In \cite{benoist2018natural} the natural parametrisation of $\SLE_\kappa$ for $\kappa \in (0,4)$ was constructed using the Gaussian free field (GFF). More precisely, it was constructed as a conditional expectation of a certain quantum length measure on $\SLE_\kappa$. The same was done in \cite{ms2022volume} in the case of $\kappa \in (4,8)$, when the curves are non-simple. The goal of this paper is to complete this picture by providing the corresponding construction in the case $\kappa = 4$. The difference in this case compared to $\kappa \neq 4$, is that in order to cut and weld Liouville quantum gravity (LQG) surfaces with $\eta \sim \SLE_4$, one must choose the LQG parameter $\gamma = 2$, which is the critical value. When working with critical LQG, the calculations turn out to be a bit more cumbersome (due to the factor containing the logarithm in the definition, see~\eqref{eq:2LQG_length}) and one has to be extra careful, as the LQG boundary length of a (Euclidean) bounded interval no longer has finite expectation.

For a simply connected domain $D$, denote by $r_D(z)$ the conformal radius of $D$ as seen from $z \in D$.

\begin{thm}\label{thm:mainresult}
Let $\eta \sim \SLE_4$ and $(f_t)$ be its centred Loewner chain. Let $h^0$ be a zero-boundary GFF on $\UH$ independent of $\eta$ and for each $t > 0$, define $h^t = h^0 \circ f_t^{-1} + 2 \log | (f_t^{-1})'|$. We define the measure $\mu^0$ on $\eta$ by
\begin{align*}
    \mu^0|_{\eta([0,t])}(dz) = F(z) \E[ \nu_{h^t} \, | \, \eta] \circ f_t(dz),
\end{align*}
where $F(z) = r_\UH(z)^{-1/2}$ and $\nu_{h^t}$ is the critical LQG boundary length measure with respect to $h^t$ (see~\eqref{eq:2LQG_length}). Then, there exists a (deterministic) constant $C > 0$ such that a.s.\ $C \mu^0$ is the natural parametrisation of $\eta$.
\end{thm}

\subsection*{Related work}
There have been many results concerning natural measures on random fractals. We already mentioned the results on the natural parametrisation of $\SLE$ in \cite{ls2011natural,lz2013natural,lr2015minkowski,benoist2018natural,ms2022volume}. Moreover, in \cite{ms2022volume}, a natural conformally covariant measure on $\CLE_\kappa$, $\kappa \in (8/3,8)$, was constructed and proved to be unique (up to multiplicative constant). Moreover, in \cite{cl2022natural}, natural conformally covariant measures on several fractals (cut points and boundary touching points of $\SLE_\kappa$, $\kappa > 4$, $\CLE$ pivotal points and carpet/gasket) were constructed and in \cite{kms2023nonsimple} the natural measure on cut points of $\SLE_\kappa$ for $\kappa > 4$ was studied further and proved to have bounded moments.

Another related result is the construction of a family of random measures on the so-called two-valued sets (TVS) of the GFF, carried out in \cite{ssv2020tvsdim}. This is done similarly to the measures above, using the imaginary multiplicative chaos (that is, a version of LQG where the real parameter $\gamma$ is replaced by $i \sigma$ for some $\sigma \in (0,\sqrt{2})$) and it was shown that if the conformal Minkowski contents of the TVS exist, then they are equal (up to deterministic constants) to the constructed measures.

\subsection*{Outline}
Section~\ref{sec:preliminaries} contains the necessary preliminaries and in Section~\ref{sec:natural_measure} we prove the main theorem. The latter section begins with the proof of the consistency of the definition of $\mu^0$ for different $t > 0$, after which the rest of the section is divided into two subsections. In Section~\ref{sec:local_finiteness_UH} we prove that $\mu^0$ is almost surely locally finite as a measure on $\UH$ and in Section~\ref{sec:conformal_covariance} we prove that $\mu^0$ is conformally covariant and use this together with the local finiteness of $\mu^0$ on $\UH$ to deduce the local finiteness of $\mu^0$ as a measure on $\SLE_4$. The main result then follows.

\subsection*{Notation}
For any quantities $a,b$, we write $a \lesssim b$ to mean $a \leq C b$ for some constant $C > 0$ which does not depend on any of the parameters of interest. Moreover, we write $a \gtrsim b$ if $b \lesssim a$. Finally, we write $a \asymp b$ if $a \lesssim b$ and $a \gtrsim b$.

We denote by $\R$ the real line and $\C$ the complex plane. Moreover, we write $\UH$ for the upper half-plane $\{z \in \C\colon \im(z) > 0\}$ and $\D$ for the unit disk $\{ z \in \C\colon |z| < 1 \}$.

We denote two-dimensional Lebesgue measure by $dz$.

\subsection*{Acknowledgements}
We thank Hao Wu for valuable input and an anonymous referee for helpful comments. LS was supported by the Finnish Academy Centre of Excellence FiRST and the ERC starting grant 804166 (SPRS). 

\section{Preliminaries}
\label{sec:preliminaries}

\subsection{Random measures}
\label{sec:random_measures}
A random measure is a random element in a space of measures on some Borel space, in our case $\C$. For a random measure $\xi$, we define the intensity $\E[\xi]$ of $\xi$ by $\E[\xi](A) = \E[\xi(A)]$ for all Borel sets $A \subset \C$. Moreover, for a $\sigma$-algebra $\G$, we define the conditional intensity of $\xi$ given $\G$ to be the random measure given by $\E[\xi \, | \, \G] (A) = \E[\xi(A) \, | \, \G]$ for each Borel set $A \subset \C$. 
For more on random measures, see \cite{kallenberg2017rmbook} (note that while they require local finiteness as a part of the definition, this property will be shown to hold for the measure we construct).

\subsection{Schramm-Loewner evolution}
\label{sec:SLE}
Consider the Loewner differential equation
\begin{align}\label{eq:LDE}
    \partial_t g_t(z) = \frac{2}{g_t(z) - W_t}, \quad g_0(z) = z,
\end{align}
with $W_t = \sqrt{\kappa} B_t$ where $\kappa > 0$ and $B_t$ is a standard Brownian motion. There exists a solution $g_t(z)$ to~\eqref{eq:LDE} for each time $t \in [0,T_z)$ where $T_z = \inf\{ t \geq 0\colon g_t(z) - W_t = 0 \}$. We let $K_t = \{ z \in \UH\colon T_z \leq t \}$ and note that $(g_t)_{t \geq 0}$ is a family of conformal maps $g_t\colon \UH \setminus K_t \to \UH$ such that $\lim_{z \to \infty} g_t(z) - z = 0$, called the $\SLE_\kappa$ Loewner chain.  By \cite{rs2005basic} (for $\kappa \neq 8$) and \cite{lsw2004lerwust} (for $\kappa = 8$) there exists a continuous curve $\eta$ such that $K_t$ is the set of points separated from infinity by $\eta([0,t])$. This curve $\eta$ is an $\SLE_\kappa$ in $\UH$ from $0$ to $\infty$. The behaviour of $\eta$ is heavily influenced by the value of $\kappa$: if $\kappa \leq 4$, $\eta$ is a.s.\ simple, if $\kappa \in (4,8)$, $\eta$ a.s.\ intersects its past as well as the boundary and if $\kappa \geq 8$, the  $\eta$ is a.s.\ space-filling. It turns out that $\eta$ has a.s.\ Hausdorff dimension $d_\kappa = 1 + \kappa/8$ for $\kappa \in (0,8)$, see \cite{beffara2008dimension}.

Denote by $(f_t)_{t \geq 0}$ the centred Loewner chain of $\eta \sim \SLE_\kappa$, that is, for each $t \geq 0$ and all $z \in \UH$, $f_t(z) = g_t(z) - W_t$. An important property of $\SLE$ is that the law of $\eta$ is scale invariant and satisfies the domain Markov property, that is, for any a.s.\ finite stopping time $\tau$, the law of $\eta^\tau(u) = f_\tau(\eta(\tau+u))$ is $\SLE_\kappa$.


\subsubsection{Natural parametrisation of $\SLE$}
We now recall some basic facts about the natural parametrisation of $\SLE$. The natural parametrisation of $\SLE$ was constructed in \cite{ls2011natural,lz2013natural} and in \cite{lr2015minkowski} shown to be equal to (a deterministic constant times) the Minkowski content of $\SLE$ which they also proved exists (we shall discuss briefly what this means below). We recall that the $d$-dimensional Minkowski content of a set $A \subset \UH$ is (if it exists) defined as the limit
\begin{align*}
    \cont_d(A) = \lim_{r \to 0} r^{d-2} \area(\{z \in \UH\colon \dist(z,A) \leq r \}).
\end{align*}
For $\eta \sim \SLE_\kappa$ we let $\CM$ denote the $d_\kappa$-dimensional Minkowski content of $\eta$ (where $d_\kappa = 1 + \kappa/8$), that is, $\CM$ is the measure $\CM(D) = \cont_{d_\kappa}(\eta \cap D)$. Typically, the Minkowski content of $\SLE$ is viewed as a parametrisation of $\eta$, that is, one can parametrise $\eta$ so that $\CM(\eta([s,t])) = t - s$ for all $0 < s < t$.

Furthermore, it is proved that if we denote by $(f_t)$ the centred Loewner chain for $\eta$, then
\begin{align*}
    &\E[ \CM(D) ] = \int_D G_\kappa(z) dz, \quad \E[ \CM(D)^2 ] = \iint_{D^2} G_\kappa(z,w) dz dw, \\
    &\E[ \CM(D) \, | \, \eta([0,t]) ] = \CM|_{\eta([0,t])}(D) + \int_D |f_t'(z)|^{2-d} G_\kappa(f_t(z)) dz,
\end{align*}
where $G_\kappa(z) = c_\kappa \sin^{\tfrac{8}{\kappa}-1}(\arg z) \im(z)^{d-2}$ and $c_\kappa > 0$ and
\begin{align*}
    G_\kappa(z,w) = \lim_{r \to 0} r^{2(d-2)} \p( \dist(z,\eta) \leq r, \ \dist(w,\eta) \leq r).
\end{align*}
Next, we recall a characterisation of the natural parametrisation which we will use to identify the measure that we will construct with the natural parametrisation. Again, let $(f_t)$ be the centred Loewner chain of $\eta \sim \SLE_\kappa$ and for each $t > 0$, let $\psi_t = f_t^{-1}$. For each $t > 0$, we define the ``unzipped'' curve $\eta^t$ by
\begin{align}\label{eq:unzipped_curve}
    \eta^t(u) = f_t(\eta(t+u)), \quad u > 0,
\end{align}
write $\eta^0 = \eta$ and recall that $\eta^t \sim \SLE_\kappa$. Moreover, for a $d_\kappa$-dimensional volume measure $\mu$ on $\eta$, we define
\begin{align}\label{eq:unzipped_measure}
    \mu^t(dz) = | \psi_t(z)|^{-d_\kappa} \mu
    \circ \psi_t(dz),
\end{align}
and note that $\mu^0 = \mu$. The following uniqueness result was proved in \cite{ls2011natural}.
\begin{thm}\label{thm:uniqueness_natural_parametrisation}
Fix $\kappa \in (0,8)$, set $d_\kappa = 1 + \kappa/8$, let $\eta \sim \SLE_\kappa$ and let $\mu$ be an a.s.\ locally finite measure on $\eta$. If $(\mu^t,\eta^t) \overset{d}{=} (\mu,\eta)$ for all $t \geq 0$, then there is a deterministic constant $C > 0$ such that $C \mu$ is the natural parametrisation of $\eta$.
\end{thm}

\subsection{Gaussian free field}
\label{sec:GFF}
Let $D$ be a Jordan domain and let $H_0(D)$ be the Hilbert space closure of the space $C_0^\infty(D)$ of smooth functions with support compactly contained in $D$, under the Dirichlet inner product
\begin{align*}
    (f,g)_\nabla = \frac{1}{2\pi} \int_D \nabla f(z) \cdot \nabla g(z) dz.
\end{align*}
Let $(\phi_n)_{n \geq 1}$ be an orthonormal basis of $H_0(D)$ and consider a sequence $(\alpha_n)_{n \geq 1}$ of i.i.d.\ $N(0,1)$ random variables. The zero-boundary GFF $h$ on $D$ is defined by the sum $h = \sum_{n \geq 1} \alpha_n \phi_n$. The law of $h$ does not depend on the choice of orthonormal basis and it is conformally invariant in the sense that if $\varphi\colon \wt{D} \to D$ is a conformal map, then $\wt{h} = h \circ \varphi$ is a zero-boundary GFF in $\wt{D}$.

The zero-boundary GFF satisfies a domain Markov property. Indeed, if $U \subset D$ is open, then we have the orthogonal decomposition $H_0(D) = H_0(U) \oplus H_0^\perp(U)$, where $H_0^\perp(U)$ is the space of functions $f \in H_0(D)$ which are harmonic on $U$. It follows that we can decompose $h$ as $h = h_U + h_U^\perp$, where $h$ is a zero-boundary GFF on $U$ and $h_U^\perp$ is a distribution which agrees with $h$ on $D \setminus U$, is harmonic on $U$ and is independent of $h_U$. We think of $h_U^\perp|_U$ as the harmonic extension of the the values of $h$ on $\partial U$ to $U$. With this in mind, one can also define a GFF with boundary data $F$ as $h+f$, where $f$ is the harmonic extension of $F$ to $D$.

One can, equivalently, define the zero-boundary GFF as a centred Gaussian process $h: H_0(D) \to L^2(D)$ with correlation kernel given by the Dirichlet Green's function for $D$. This means that $(h,f)$, $f \in H_0(D)$, is a collection of centred Gaussians with covariance given by $\E[(h,f) (h,g)] = \int_{D \times D} f(z) G_D(z,w) g(w) dz dw$, where $G_D$ is the Green's function on $D$ for Dirichlet boundary data.

We let $H(D)$ denote the Hilbert space closure of the set of functions $f \in C^\infty(D)$ with $(f,f)_\nabla < \infty$ such that $\int_D f dz = 0$ (we do not require that the functions have compact support), with respect to $(\cdot,\cdot)_\nabla$. A free boundary GFF is defined in the same way as a zero-boundary GFF, with an orthonormal basis of $H(D)$ replacing that of $H_0(D)$. Similarly, the law of the free boundary GFF is independent of the choice of orthonormal basis and it is conformally invariant.

In the interior of the domain, the laws of a zero-boundary GFF and a free boundary GFF are mutually absolutely continuous. In fact, we may decompose a free boundary GFF $h^f$ as $h^f = h^0 + \Fh$, where $h^0$ is a zero-boundary GFF and $\Fh$ is a harmonic function, independent of $h^0$.

\begin{rmk}[Radial/lateral decomposition in $\UH$]
We have the orthogonal decomposition $H(\UH) = H_\Rad(\UH) \oplus H_\Lat(\UH)$, where $H_\Rad(\UH)$ (resp.\ $H_\Lat(\UH)$) is the space of functions in $H(\UH)$ which are constant (resp.\ has mean zero) on each semicircle $\partial B(0,s)$.

If $h$ is a free boundary GFF on $\UH$, then the projection of $h$ onto $H_\Rad(\UH)$ is called the radial part of $h$ and if $h_\Rad(z) = h_{|z|}(0)$ denotes the average value of $h$ on $\partial B(0,|z|)$, then $(h_\Rad(e^{-t}))_{t \in \R}$ has the same law as $(B_{2t})_{t \in \R}$ where $B$ is a two-sided Brownian motion with $B_0 = 0$. The projection of $h$ onto $H_\Lat(\UH)$ is called the lateral part of $h$.
\end{rmk}

\subsection{Liouville quantum gravity}
\label{sec:LQG}
Fix $\gamma \in (0,2]$. A $\gamma$-Liouville quantum gravity ($\gamma$-LQG) surface is a law on equivalence classes of pairs $(D,h)$ where $D$ is a simply connected domain and $h$ a distribution on $D$, such that $(D,h)$ and $(\wt{D},\wt{h})$ are equivalent if there exist a conformal map $\psi\colon \wt{D} \to D$ such that 
\begin{align}\label{eq:LQG_coordinate_change}
    \wt{h} = h \circ \psi + Q_\gamma \log|\psi'|, \quad Q_\gamma = \frac{2}{\gamma} + \frac{\gamma}{2}.
\end{align}
Typically, $h$ is a random distribution which looks locally like a GFF. One can also define quantum surfaces with marked points. We say that $(D,h,z_1,\dots,z_n)$ and $(\wt{D},\wt{h},\wt{z}_1,\dots,\wt{z}_n)$ are equivalent if $(D,h)$ and $(\wt{D},\wt{h})$ are equivalent as quantum surfaces and $\psi(\wt{z}_j) = z_j$ for all $1 \leq j \leq n$.

An LQG surface comes naturally equipped with an area measure and a length measure, but we shall only need the latter. Consider a $2$-LQG surface embedded into $\UH$ and for $\epsilon > 0$ and $x \in \R$, we denote by $h_\epsilon(x)$ the average value of $h$ on $\partial B(x,\epsilon) \cap \UH$ (this makes sense for any $h$ which is locally absolutely continuous with respect to a GFF). Then we define the boundary length measure to be
\begin{align}\label{eq:2LQG_length}
    \nu_h(dx) = \lim_{\epsilon \to 0} \epsilon \!\left(\log (1/\epsilon) -\frac{h_\epsilon(x)}{2} \right)\! e^{h_\epsilon(x)} dx.
\end{align}
By the conformal coordinate change~\eqref{eq:LQG_coordinate_change} one can then define the quantum length of boundaries in arbitrary simply connected domains. We note that while $\nu_h$ is a measure on $\partial D$, it can actually be used to measure the length of curves inside on $D$ as well. Indeed, if $\eta$ is a curve in $D$ and $f\colon D \setminus \eta \to D$ is a conformal map which extends to the boundary (in the sense of prime ends), then letting $\eta^L$ (resp.\ $\eta^R$) denote the left (resp.\ right) side of $\eta$ we define the length of $\eta^q$ as $\nu_{\wt{h}}(f(\eta^q))$, where $\wt{h} = h \circ f^{-1} + 2 \log|(f^{-1})'|$ and $q \in \{L,R\}$ (note here that $Q_2 = 2$). 

Next, we recall the definition of a quantum wedge. 
\begin{defin}
A $(\gamma,\alpha)$-quantum wedge, $\alpha \in (-\infty,Q_\gamma)$ is a doubly marked $\gamma$-LQG surface $\CW = (\UH,h,0,\infty)$ such that the projection $h_\Lat$ of $h$ on $H_\Lat(\UH)$ has the law of the lateral part of a free boundary GFF on $\UH$ and such that if $X_t$ denotes the average value of $h$ on $\partial B(0,e^{-t})$, then $X$ is as follows.
\begin{itemize}
    \item $(X_t)_{t \geq 0}$ has the law of $(B_{2t} + \alpha t)_{t \geq 0}$, where $B$ is a standard Brownian motion with $B_0 = 0$, conditioned so that $B_{2t} + \alpha t \leq Q_\gamma t$ for all $t \geq 0$.
    \item $(X_t)_{t \leq 0}$ has the law of $(\wh{B}_{-2t} + \alpha t)_{t \leq 0}$, where $\wh{B}$ is a standard Brownian motion with $\wh{B}_0 = 0$.
    \item $h_\Lat$, $(X_t)_{t \leq 0}$ and $(X_t)_{t \geq 0}$ are independent.
\end{itemize}
We then write $(\UH,h,0,\infty) \sim \qwedgeA{\gamma}{\alpha}$. The particular embedding in this definition is called the last exit parametrisation.
\end{defin}

\begin{rmk}\label{rmk:wedge_strip}
    Sometimes it is useful to parametrise a quantum wedge by the strip instead. This is done using the conformal coordinate change~\eqref{eq:LQG_coordinate_change} with $\psi(z) = e^{-z}$. Then, a last exit parametrisation embedding of $(h,\strip,-\infty,+\infty)$ can be sampled by letting $h = X + h_\Lat$ where $h_\Lat$, the projection of $h$ onto $H_\Lat(\strip)$, has the law of the lateral part of a free boundary GFF in $\strip$ and $X_t$ is the average value of $h$ on $\{t \} \times [0,\pi]$, defined as follows.
    \begin{itemize}
        \item $(X_t)_{t \geq 0}$ has the law of $(B_{2t} - (Q_\gamma - \alpha)t)_{t \geq 0}$, where $B$ is a standard Brownian motion with $B_0 = 0$, conditioned so that $B_{2t} - (Q_\gamma - \alpha) t \leq 0$ for all $t \geq 0$.
        \item $(X_t)_{t \leq 0}$ has the law of $(\wt{B}_{-2t} - (Q_\gamma - \alpha)t)_{t \leq 0}$, where $\wt{B}$ is a standard Brownian motion with $\wt{B}_0 = 0$.
        \item $h_\Lat$, $(X_t)_{t \leq 0}$ and $(X_t)_{t \geq 0}$ are independent.
    \end{itemize}
\end{rmk}

\begin{rmk}\label{rmk:bm_conditioned}
    Some care has to be taken when conditioning on a zero probability event. Fix $a > 0$ and let $(U_t)_{t \geq 0}$ be $(B_t-at)_{t \geq 0}$ ``conditioned to stay negative''. By this we mean the weak limit as $\epsilon \to 0$ of the processes $(B_t-at)_{t \geq 0}$ conditioned to stay below $\epsilon$. A sample from this law can be drawn by sampling a standard Brownian motion $\wh{B}_t$, letting $\tau = \sup\{t \geq 0: \wh{B}_t - at \geq 0\}$ (which is a.s.\ finite) and setting $U_t = \wh{B}_{\tau + t} - a(\tau + t)$.
\end{rmk}

\section{Natural measure}\label{sec:natural_measure}
For any field $h$ we let $\nu_h$ denote the $2$-LQG boundary measure associated with $h$, defined in~\eqref{eq:2LQG_length}. Moreover, we let $\eta$ be an $\SLE_4$ in $\UH$ from $0$ to $\infty$ and $(f_t)$ its centred Loewner chain. For $t \geq 0$ we denote by $\eta^t$ the curve $\eta^t(u) = f_t(\eta(t+u))$ and note that $\eta^0 = \eta$. Moreover, we write $\psi_t = f_t^{-1}$ and let $f_{s,t} = f_t \circ \psi_s$ and $\psi_{s,t} = f_{s,t}^{-1} = f_s \circ \psi_t$.

We let $h^0$ be a zero-boundary GFF independent of $\eta$ and for each $t > 0$ let 
\begin{align}\label{eq:zipping_fields}
    h^t = h^0 \circ \psi_t + 2 \log |\psi_t'|
\end{align}
be the field on $\UH$ formed by unzipping $h^0$ along $\eta^0$. Then, we define the measure $\mu^0$ on $\eta^0$ by
\begin{align}\label{eq:measure_eta_0}
    \mu^0|_{\eta^0([0,t])}(dz) = F(z) \E[ \nu_{h^t} \, | \, \eta^0] \circ f_t(dz),
\end{align}
where $F(z) = r_\UH(z)^{-1/2}$, and $r_\UH(z)$ denotes the conformal radius of $\UH$ seen from $z$. An identity that we will use several times is the following. If $\varphi$ be a conformal map, then
\begin{align}\label{eq:inverse_derivative}
    \varphi'(\varphi^{-1}(z)) = \frac{1}{(\varphi^{-1})'(z)}.
\end{align}

We begin by proving that this definition of $\mu^0$ is consistent for different values of $t$.
\begin{lem}
The definition of $\mu^0$ given in~\eqref{eq:measure_eta_0} is consistent for different values of $t$. That is, if $0 < s < t$, then for each $A \subset [0,s]$, we a.s.\ have that
\begin{align*}
    \int_{\eta^0(A)} F(z) \E[ \nu_{h^s} \, | \, \eta^0] \circ f_s(dz) = \int_{\eta^0(A)} F(z) \E[ \nu_{h^t} \, | \, \eta^0] \circ f_t(dz).
\end{align*}
\end{lem}
\begin{proof}
If $0 < s < t$, then $h^t = h^s \circ \psi_{s,t} + 2 \log |\psi_{s,t}'|$. Thus,
\begin{align*}
    \nu_{h^t} \circ f_t(dz) &= \nu_{h^s \circ \psi_{s,t} + 2\log|\psi_{s,t}'|} \circ f_t(dz) \\
    &= \left[ \lim_{\epsilon \to 0} \epsilon \! \left( \log(1/\epsilon) - \frac{(h^s \circ \psi_{s,t})_\epsilon + 2 \log| \psi_{s,t}'|}{2} \right) e^{(h^s \circ \psi_{s,t})_\epsilon + 2 \log | \psi_{s,t}'|} \right] \circ f_{s,t} \circ f_s(dz) \\
    &= \left[ \lim_{\epsilon \to 0} \epsilon \! \left( \log(1/\epsilon) - \frac{h^s_{\epsilon/|f_{s,t}'|} - 2 \log| f_{s,t}'|}{2} \right) e^{h^s_{\epsilon/|f_{s,t}'|} - 2 \log | f_{s,t}'|} \right] \circ f_s(dz) \\
    &= \left[ \lim_{\epsilon \to 0} \frac{\epsilon}{|f_{s,t}'|} \! \left( \log(|f_{s,t}'|/\epsilon) - \frac{h^s_{\epsilon/|f_{s,t}'|}}{2} \right) e^{h^s_{\epsilon/|f_{s,t}'|}} \right] \circ f_s(dz) \\
    &= \nu_{h^s} \circ f_s(dz),
\end{align*}
where in the third equality we used~\eqref{eq:inverse_derivative} and the fact that as $\epsilon \to 0$, the preimages of circles of radius $\epsilon$ around $f_t(z)$ under $f_{s,t}$ are (roughly) circles of radius $\epsilon/|f_{s,t}'(f_s(z))|$ around $f_s(z)$. This concludes the proof.
\end{proof}

The rest of this section is divided into two subsections, the first of which is focused on establishing that $\mu^0$ is locally finite as a measure on $\UH$, that is, for each compact $K \subset \UH$, we have that $\mu^0(K) < \infty$ a.s. The second subsection is devoted to proving the conformal covariance of $\mu^0$ and then to deducing the local finiteness of $\mu^0$ as a measure on $\eta$, that is, for each compact $I \subset (0,\infty)$, we have that $\mu^0(\eta^0(I)) < \infty$ a.s.

\subsection{Local finiteness on $\UH$}
\label{sec:local_finiteness_UH}

\begin{lem}\label{lem:local_finiteness_UH}
Almost surely, $\mu^0$ is locally finite as a measure on $\UH$. That is, $\p[\mu^0(K) < \infty] = 1$ for each $K \subset \UH$ compact.
\end{lem}

We begin by noting that we can decompose a $(2,1)$-quantum wedge as the sum of a zero-boundary GFF and a harmonic function. For the proof, we refer to \cite{kms2023nonsimple} (the parametrisation is slightly different but the proof is the same).

\begin{lem}\label{lem:wedge_decomposition}
Let $(\UH,h^w,0,\infty) \sim \qwedgeA{2}{1}$ have the last exit parametrisation. Then, we can write $h^w = h^0 + \Fh$, where $h^0$ is a zero-boundary GFF on $\UH$ and $\Fh$ is a harmonic function on $\UH$ which is independent of $h^0$.
\end{lem}

The following moment bound is the key ingredient in the proof of Lemma~\ref{lem:local_finiteness_UH}. The reason for the choice of quantum surface is that $\qwedgeA{2}{1}$ is natural in the context of $\SLE_4$, as if we unzip such a surface along an $\SLE_4$ curve for a finite time, then the resulting law is still that of a $\qwedgeA{2}{1}$, see \cite[Theorem~1.5]{hp2021welding}.

\begin{lem}\label{lem:moment_finite}
Let $\eta^0 \sim \SLE_4$ and let $(\UH,h^w,0,\infty) \sim \qwedgeA{2}{1}$ have the last passage parametrisation and be independent of $\eta^0$. Set $h^{w,t} = h^w \circ \psi_t + 2 \log|\psi_t'|$. Then there exists $p \in (0,1)$ such that for each $t \geq 0$,
\begin{align*}
    \E[ \nu_{h^{w,t}}(f_t(\D \cap \UH))^p] < \infty.
\end{align*}
\end{lem}

Before proving Lemma~\ref{lem:moment_finite}, we begin by providing some moment bounds.

\begin{lem}\label{lem:moment_bound_strip}
Let $(\strip,h,-\infty,+\infty) \sim \qwedgeA{2}{1}$ have the last exit parametrisation and let for each $k \in \Z$, $I_k = [k-1,k] \times \{0,\pi\}$. For each $p \in (0,1)$ there exists a constant $C_p > 0$ such that if $k \geq 1$, then $\E[\nu_h(I_k)^p] \leq C_p$ and if $k \leq 0$, then 
\begin{align*}
    \E[\nu_h(I_k)^p] \leq C_p e^{-(4p^2 + 2p)k}.
\end{align*}
\end{lem}
\begin{proof}
We let $h = X + h_\Lat$ be the decomposition into the average on vertical lines process $X$ and the lateral part $h_\Lat$ of $h$. By~\cite[Lemma~A.4]{kms2021regularity} we have that $\E[ \nu_{h_\Lat}(I_k)^p] < \infty$ for all $k \in \Z$ and $p \in (0,1)$. Moreover, we have that
\begin{align}\label{eq:measure_bound}
    \nu_h(I_k) \leq \nu_{h_\Lat}(I_k) \exp\!\Big(2 \sup_{t \in [k-1,k]} X_t \Big).
\end{align}
Since $X_t$ is non-positive whenever $t \geq 0$, the first assertion follows. Consider now the case $k \leq 0$, so that the Brownian motion in the process $X$ is not conditioned on any event. Then,
\begin{align*}
    \sup_{t \in [k-1,k]} X_t \leq 1-k + \sup_{t \in [k-1,k]} \wt{B}_{2t} = 1-k + \wt{B}_{-2k} + \sup_{t \in [k-1,k]} \wt{B}_{-2t} - \wt{B}_{-2k}.
\end{align*}
Thus, since $\wt{B}$ has stationary and independent increments and for a Brownian motion $\wh{B}$ and any $c>0$, the expectation $\E[ \exp(c \sup_{t \in [0,2]} \wh{B}_t)]$ is finite, it follows from~\eqref{eq:measure_bound} that for $p \in (0,1)$,
\begin{align*}
    \E[\nu_h(I_k)^p] &\leq \E[\nu_{h_\Lat}(I_k)^p] \E\!\left[\exp\!\left(2p \sup_{t \in [k-1,k]} X_t \right) \right] \\
    &= e^{2p(1-k)} \E[\nu_{h_\Lat}(I_k)^p] \E[ \exp(2p B_{-2k})] \E\!\left[\exp\!\left(2p \sup_{t \in [0,2]} \wh{B}_t \right) \right] \\
    &\leq C_p e^{-(4p^2 + 2p)k}
\end{align*}
where $C_p$ is some positive constant depending only on $p$.
\end{proof}

\begin{proof}[Proof of Lemma~\ref{lem:moment_finite}]
Let $\tau = \sup\{t \geq 0: \eta^0(t) \in \D\}$. We note that by conformal invariance,
\begin{align*}
    \nu_{h^{w,s}}(f_s(\D \cap \UH)) \leq \nu_{h^{w,\tau}}(f_\tau(\D \cap \UH)) = \nu_{h^{w,t}}(f_t((\D \cap \UH))
\end{align*}    
whenever $s \leq \tau \leq t$. By \cite[Theorem~4.1]{wz2017boundaryarm} we have that
\begin{align}\label{eq:arm_exponent}
    \p \! \left[ \sup_{0 \leq t \leq \tau} | \eta^0(t)| \geq R \right] \lesssim R^{-2}.
\end{align}
We let $\tau_R = \inf\{t \geq 0: |\eta^0(t)| \geq R \}$ and note that for all $0 \leq t \leq \tau_R$, $|f_t(0^-)|,|f_t(0^+)| \leq 4R$. Indeed, this follows immediately from \cite[Equation~(8)]{schoug2020multifractal} and a comparison with the compact $\UH$-hull $R\ol{\D} \cap \UH$. We shall now bound moments of the quantum length of $[-4R,4R]$.

We first bound $p$th moment of the quantum length of $[-4R,-1] \cup [1,4R]$. As above, let $I_k = [k-1,k] \times \{0,\pi\}$. Then, letting $(\strip,h^S,-\infty,+\infty) \sim \qwedgeA{2}{1}$, this corresponds to bounding $\E[\nu_{h^S}([-\log 4R,0] \times \{0,\pi\})^p]$. By Lemma~\ref{lem:moment_bound_strip} and the inequality $(\sum_j x_j)^p \leq \sum_j x_j^p$ (for $x_j > 0$, $0<p<1$),
\begin{align*}
    \E[\nu_{h^S}([-\log 4R,0] \times \{0,\pi\})^p] &\leq \sum_{k=0}^{\lceil \log 4R \rceil} \E[\nu_{h^S}(I_{-k})^p] \leq C_p \sum_{k=0}^{\lceil \log 4R \rceil} e^{(4p^2+2p)k} \\
    &\lesssim \int_0^{\lceil \log 4R \rceil} e^{(4p^2+2p)x} dx \lesssim R^{4p^2+2p},
\end{align*}
where the implicit constants depend on $p$. Moreover, let as in Remark~\ref{rmk:bm_conditioned}, $\wh{B}$ be a standard Brownian motion, $\tau = \sup\{ t \geq 0: \wh{B}_t - t \}$ and such that $X_t = \wh{B}_{\tau+t} - (\tau+t)$, and let $M = \sup_{t \geq 0} \wh{B}_t - t$. Then, letting $X_k^* = \sup_{t \in [k-1,k]} X_t$ and $k^* = \argmax_{t \in [k-1,k]} \wh{B}_t - t$, we have that
\begin{align*}
    \E[\nu_{h^S}([0,\infty) \times \{0,\pi\})^p] &\leq \E\!\left[ \sum_{k = 1}^\infty \exp\!\left( 2p X_k^* \right) \nu_{h_\Lat}(I_k)^p \right] \lesssim \E\!\left[ \sum_{k = 1}^\infty \exp\!\left( 2p X_k^* \right) \right] \\
    &\leq  \E\!\left[ \sum_{k = 1}^\infty \exp\!\left( 2p (\wh{B}_{k^*} - k^*) \right) \right] = \E\!\left[ \E\!\left[\sum_{k = 1}^\infty \exp\!\left( 2p (\wh{B}_{k^*} - k^*) \right) \middle| M \right]\right] \\
    &\lesssim \E[\exp(2pM)] \quad (\cite[\textup{Lemma}~\textup{A}.5]{dms2021mating})
\end{align*}
Letting $T_x = \inf\{ t\geq 0: \wh{B}_t \geq x \}$, we have that $\p( M > x) = \p( T_x < \infty) = \exp(-2x)$ and hence, since $p \in (0,1)$, $\E[\exp(2pM)]$ is finite. Consequently, 
\begin{align*}
    \E[\nu_{h^S}([-\log 4R,\infty] \times \{0,\pi\})^p] \lesssim R^{4p^2 + 2p}
\end{align*}
where the implicit depends only on $p$. We let $E_n$ be the event that $\sup_{0 \leq t \leq \tau} |\eta(t)| \in (n-1,n]$. Then, by \cite[Theorem~1.5]{hp2021welding}, we have that
\begin{align*}
    \E[\nu_{h^{w,\tau}}(f_\tau(\D \cap \UH))^p] &= \sum_{n=0}^\infty \E[\nu_{h^{w,\tau}}(f_\tau(\D \cap \UH))^p \, | \, E_n]\p(E_n) \\
    &\leq \sum_{n=0}^\infty \E[\nu_{h^w}([-4n,4n])^p]\p(E_n) \\
    &\lesssim \sum_{n=1}^\infty n^{4p^2+2p} n^{-2}
\end{align*}
which is finite if we pick $p$ small enough, so that $4p^2 + 2p - 2 < -1$.
\end{proof}

\begin{proof}[Proof of Lemma~\ref{lem:local_finiteness_UH}]
Let $p \in (0,1)$ be as in Lemma~\ref{lem:moment_finite}. We shall begin by showing that $\E[ \nu_{h^t} (f_t(K))^p] \newline < \infty$ for each $K \subset \tfrac{1}{2} \ol{\D} \cap \UH$ compact and each $t > 0$.

Fix some compact $K \subset \tfrac{1}{2} \ol{\D} \cap \UH$ and let $\delta = \dist(K,\partial \UH) > 0$. We decompose $h^w = h^0 + \Fh$ as in Lemma~\ref{lem:wedge_decomposition} and note that since $\delta > 0$, we have that $\Fh$ is almost surely bounded on $K$. Thus, if $E_K$ denotes the event that $\sup_{z \in K} | \Fh(z) | \leq C$, then $\p[ E_K ] \to 1$ as $C \to \infty$. We choose $C > 0$ large enough so that $\p[E_K] \geq 1/2$. Then, it follows that
\begin{align}\label{eq:upper_bound_moment_h^t}
    \E[ \nu_{h^t} (f_t(K))^p \, | \, \Fh] \one_{E_K} \leq e^{C p} \E[ \nu_{h^{w,t}} (f_t(K))^p \, | \, \Fh] \one_{E_K} \leq e^{C p} \E[ \nu_{h^{w,t}} (f_t(K))^p \, | \, \Fh].
\end{align}
Since $\E[\E[ \nu_{h^{w,t}} (f_t(K))^p \, | \, \Fh]] = \E[ \nu_{h^{w,t}} (f_t(K))^p] < \infty$, it follows that $\E[ \nu_{h^{w,t}} (f_t(K))^p \, | \, \Fh]$ is almost surely finite. By the independence of $h^0$ and $\Fh$, we have that $\E[ \nu_{h^t} (f_t(K))^p \, | \, \Fh] = \E[ \nu_{h^t} (f_t(K))^p]$. Since the right-hand side of~\eqref{eq:upper_bound_moment_h^t} is a.s.\ finite it follows that $\E[ \nu_{h^t} (f_t(K))^p] < \infty$.

Next, we shall show that $\mu^0(K) < \infty$ almost surely. Let $\tau$ be the last exit time of $\D$ for $\eta^0$. Since $\eta^0$ is transient, we have that $\tau < \infty$ almost surely and hence that if $E_T^*$ is the event that $\tau \leq T$, then $\p[E_T^*] \to 1$ as $T \to \infty$. Thus, since $\dist(z,\UH) \geq \delta$ for all $z \in K$, we have that
\begin{align}\label{eq:truncated_moment_bound}
    \E[\mu^0(K)^p \one_{E_T^*}] \leq \delta^{-p/2} \E[\nu_{h^\tau}(f_\tau(K))^p \one_{E_T^*}] \leq \delta^{-p/2} \E[\nu_{h^T}(f_T(K))^p]< \infty.
\end{align}
Assume for the sake of contradiction that $\p[\mu^0(K) = \infty] = p_0 > 0$. Choose $T > 0$ large enough so that $\p[E_T^*] \geq 1 - p_0/2$. Then, it follows from~\eqref{eq:truncated_moment_bound} that conditional on the event $E_T^*$, $\mu^0(K)$ is a.s.\ finite, and consequently, $\{ \mu^0(K) = \infty \}$ is contained in the union of $(E_T^*)^c$ and a zero probability event. This, however, is a contradiction since $\p[(E_T^*)^c] \leq p_0/2$. Hence $\mu^0(K)$ is a.s.\ finite.

We now consider a general compact $\wt{K} \subset \UH$ and let $R = \inf\{r > 0\colon \wt{K} \subset B(0,r/2) \}$ and $\delta = \dist(\wt{K},\partial \UH) > 0$. Let $\varphi_R(z) = z/R$ and note that $K \coloneqq \varphi_R(\wt{K}) \subset \tfrac{1}{2} \ol{\D} \cap \UH$. Moreover, we note that $\varphi_R(\eta^0) \sim \SLE_4$, $\wt{h}^0 \coloneqq h^0 \circ \varphi_R^{-1}$ is a zero-boundary GFF on $\UH$. Furthermore, we let $\sigma$ be the last exit time of $B(0,R)$ for $\eta^0$ and $(\wt{f}_t)$ be the Loewner chain corresponding to the curve $\wt{\eta}^0(t) \coloneqq \varphi_R(\eta^0(t))$, that is, $\wt{f}_t(z) = \tfrac{1}{R} f_t(Rz)$, and we let $\wt{\psi}_t(z) = \wt{f}_t^{-1}(z) = \tfrac{1}{R} \psi_t(Rz)$ and $\wt{h}^t = \wt{h}^0 \circ \wt{\psi}_t + 2 \log | \wt{\psi}_t'|$. Then by the LQG coordinate change and the scale invariance of $\eta^0$ and $h^0$, it follows that a.s.
\begin{align*}
    \nu_{h^\sigma}(f_\sigma(\wt{K})) =  \nu_{\wt{h}^\sigma + 2 \log R}(\wt{f}_\sigma(K)) \overset{d}{=} \nu_{h^\tau + 2\log R}(f_\tau(K)) = R^2 \nu_{h^\tau}(f_\tau(K)).
\end{align*}
Thus, it follows analogously to the above that $\E[\nu_{h^T}(f_T(\wt{K}))^p]$ is almost surely finite for $T>0$ and hence that $\mu^0(\wt{K})$ is as well.
\end{proof}

\subsection{Conformal covariance}
\label{sec:conformal_covariance}
Fix $s>0$ and let $\wt{h}^s$ be a zero-boundary GFF on $\UH$, independent of $\eta^0$. This then gives us a field $\wt{h}^0 = \wt{h}^s \circ f_s + 2 \log |f_s'|$ by zipping up along $\eta^0$, and hence a family of fields $(\wt{h}^t)_{t \geq 0}$ by letting $\wt{h}^t = \wt{h}^0 \circ \psi_t + 2 \log|\psi_t'|$. We recall that $f_{s,t} = f_t \circ \psi_s\colon \UH \setminus \eta^s([0,t-s]) \to \UH$ and define the measure $\wt{\mu}^s$ by
\begin{align}
    \wt{\mu}^s|_{\eta^s([0,t-s])}(dz) = F(z) \E[ \nu_{\wt{h}^t} \, | \, \eta^s] \circ f_{s,t}(dz).
\end{align}
Recall from~\eqref{eq:unzipped_measure} that $\mu^t(dz) = |\psi_t'(z)|^{-d_4} (\mu^0 \circ \psi_t)(dz)$, where $d_4 = 3/2$.
\begin{lem}\label{lem:conformal_covariance}
Almost surely, $\wt{\mu}^s = \mu^s$, that is,
\begin{align*}
    \mu^0|_{\eta^0([s,\infty))} \circ \psi_s(dz) = |\psi_s'(z)|^{3/2} \wt{\mu}^s(dz).
\end{align*}
\end{lem}
\begin{proof}
Since $\wt{h}^s$ is a zero-boundary GFF in $\UH$, we have that $\wt{h}^s \circ f_s$ is a zero-boundary GFF in $\UH \setminus \eta^0([0,s])$. Thus we may assume that $\wt{h}^s$ and $h^0$ are coupled together in such a way that we can define a Gaussian field $\FH$, which is conditionally independent of $\wt{h}^s$ given $\eta^0$ and such that $h^0 = \wt{h}^s \circ f_s + \FH$. Then the covariance kernel of $\FH$ is given by
\begin{align*}
    G_\UH(z,w) - G_{\UH \setminus \eta^0([0,s])}(z,w) = G_\UH(z,w) - G_\UH(f_s(z),f_s(w))
\end{align*}
and hence the variance of $\FH$ at a point $z$ is well-defined and equal to
\begin{align*}
    \Var \, \FH(z) = \lim_{\epsilon \to 0} \Var \, \FH_\epsilon(z) = \log r_\UH(z) - \log r_{\UH}(f_s(z)) + \log |f_s'(z)|
\end{align*}
(here $\FH_\epsilon(z)$ denotes the average value of $\FH$ on the circle $\partial B(z,\epsilon)$). The term $\log |f_s'(z)|$ comes from the change of variables dilating the ball of radius $\epsilon$ roughly by a factor of $|f_s'(z)|$.

We note that $h^t = h^0 \circ \psi_t + 2 \log |\psi_t'|=\wt{h}^t + \FH \circ \psi_t +2 \log |\psi_s'|\circ \psi_{s,t}$ and that since $\FH$ is centred and $\Var \, \FH(z)$ is finite, it follows from \cite[Lemma~2.1]{powell2021critical} that
\begin{align*}
    \E\! \left[ \lim_{\epsilon \to 0} \epsilon \!\left(\frac{(\FH \circ \psi_t)_\epsilon}{2} + \log |\psi_s'| \circ \psi_{s,t} \right) e^{\wt{h}^t_{\epsilon}+(\FH\circ \psi_t)_{\epsilon}+2 \log |\psi_s'|\circ \psi_{s,t}} \, \middle| \, \eta^0 \right] \circ f_t(dz)
\end{align*}
is the zero measure. It follows that
\begin{align*}
    &\mu^{0}|_{\eta^0([0,t])}(dz) = \\
    &=F(z)\E \! \left[\lim_{\epsilon \to 0} \epsilon \!\left(\log \!\left(\frac{1}{\epsilon} \right)-\frac{\wt{h}^t_{\epsilon}+(\FH\circ \psi_t)_{\epsilon}+2 \log |\psi_s'|\circ \psi_{s,t}}{2} \right)e^{\wt{h}^t_{\epsilon}+(\FH\circ \psi_t)_{\epsilon}+2 \log |\psi_s'|\circ \psi_{s,t}} \, \middle| \, \eta^0 \right]\circ f_t(dz) \\
    &=F(z)\E \!\left[\lim_{\epsilon \to 0} \epsilon \! \left(\log \left(\frac{1}{\epsilon} \right)-\frac{\wt{h}^t_\epsilon}{2} \right)|\psi_s'|^2\circ\psi_{s,t}e^{\wt{h}^t_\epsilon+(\FH\circ \psi_t)_\epsilon} \, \middle| \, \eta^0 \right]\circ f_t(dz) \\
    &= F(z) |f_s(z)|^{-2} \E[ e^{\FH(z)} \, | \, \eta^0] \E \!\left[\lim_{\epsilon \to 0} \epsilon \! \left(\log \left(\frac{1}{\epsilon} \right)-\frac{\wt{h}^t_\epsilon}{2} \right) e^{\wt{h}^t_\epsilon} \, \middle| \, \eta^0 \right]\circ f_t(dz),
\end{align*}
where~\eqref{eq:inverse_derivative} was used in the last inequality. Since the conditional law of $\FH(z)$ given $\eta^0$ is that of a centred Gaussian, we have that
\begin{align*}
    \E[ e^{\FH(z)} \, | \, \eta^0] = r_\UH(z)^{1/2} r_\UH(f_s(z))^{-1/2} |f_s'(z)|^{1/2} = \frac{F(f_s(z))}{F(z)} |f_s'(z)|^{1/2}
\end{align*}
and hence
\begin{align*}
    \mu^{0}|_{\eta^0([0,t])}(dz) = F(f_s(z)) |f_s'(z)|^{-3/2} \E \!\left[\lim_{\epsilon \to 0} \epsilon (\log (1/\epsilon)-\wt{h}^t_\epsilon/2) e^{\wt{h}^t_\epsilon} \, \middle| \, \eta^0 \right]\circ f_t(dz).
\end{align*}
Consequently,
\begin{align*}
    \mu^{0}|_{\eta^0([s,t])} \circ \psi_s(dz) &= F(z) |f_s'(\psi_s(z))|^{-3/2} \E \!\left[\lim_{\epsilon \to 0} \epsilon (\log (1/\epsilon)-\wt{h}^t_\epsilon/2) e^{\wt{h}^t_\epsilon} \, \middle| \, \eta^0 \right]\circ f_{s,t}(dz) \\
    &= |\psi_s'(z)|^{3/2} F(z) \E[ \nu_{\wt{h}^t} \, | \, \eta^s] \circ f_{s,t}(dz) \\
    &= | \psi_s'(z)|^{3/2} \wt{\mu}^s|_{\eta^s([0,t-s])}(dz),
\end{align*}
where~\eqref{eq:inverse_derivative} was used in the second equality. Thus, the proof is complete.
\end{proof}

The following lemma is proved in the same way as Lemma~\ref{lem:conformal_covariance}, and hence the proof is omitted.
\begin{lem}\label{lem:conformal_covariance_intensity}
Let $\phi_a(z) = az$ for $a > 0$. Then,
\begin{align*}
    \E[\mu^0] \circ \phi_a(dz) = a^{3/2} \E[\mu^0](dz).
\end{align*}
\end{lem}

We need that the intensity is absolutely continuous with respect to two-dimensional Lebesgue measure. In essence, that the randomness of the curve causes an averaging of the measure $\E[\mu^0]$, so that its mass is spread out over the entire $\UH$, rather than in a subset of zero Lebesgue measure. The following was proved for the measures on $\SLE_\kappa$ for $\kappa \in (4,8)$, see \cite[Lemma~3.6]{ms2022volume}, but the exact same proof works with the measure $\mu^0$ on $\SLE_4$.
\begin{lem}\label{lem:absolute_continuity_lebesgue}
The measure $\E[\mu^0]$ is absolutely continuous with respect to Lebesgue measure.
\end{lem}

Next, we recall the following lemma of \cite{ms2022volume}.
\begin{lem}[Lemma~3.7 of \cite{ms2022volume}]\label{lem:density_measure}
For each $a > 0$, let $\phi_a(z) = az$. Let $m$ be a measure on $\UH$ which is absolutely continuous with respect to Lebesgue measure and satisfies
\begin{align*}
    m \circ \phi_a(dz) = a^d m(dz)
\end{align*}
for some $d > 1$ and all $a > 0$. Then there exists some function $H(z) = H(\arg z)$ such that
\begin{align*}
    m(dz) = H(\arg z) \im(z)^{d-2} dz.
\end{align*}
\end{lem}

By Lemmas~\ref{lem:conformal_covariance_intensity} and~\ref{lem:absolute_continuity_lebesgue} the conditions of Lemma~\ref{lem:density_measure} are satisfied for $\E[\mu^0]$. Next, we note the form of $H$ in the case of $m = \E[\mu^0]$.
\begin{lem}\label{lem:density_intensity}
There is some constant $c > 0$ such that
\begin{align*}
    \E[\mu^0](dz) = c \sin(\arg z) \im(z)^{-1/2} dz.
\end{align*}
\end{lem}
\begin{proof}
This is proved exactly in the same way as \cite[Lemma~3.8]{ms2022volume}.
\end{proof}

We wrap up this section by deducing the following. The proof is the same as that of \cite[Lemma~3.9]{ms2022volume}, but it is very short so we repeat it here.
\begin{lem}\label{lem:local_finiteness_eta}
Almost surely, $\mu^0$ is locally finite. That is, almost surely,
\begin{align*}
    \mu^0(\eta([s,t])) < \infty \quad \text{for all} \quad 0 < s < t.
\end{align*}
\end{lem}
\begin{proof}
Note that by Lemma~\ref{lem:density_intensity}, $\E[ \mu^0(B(0,R))] < \infty$ for each $R>0$, so that a.s., $\mu^0(B(0,R))$ is finite. Moreover, for each $0 \leq s \leq t$, we have that $\p( \eta([s,t]) \subset B(0,R)) \to 1$ as $R \to \infty$. Finally, since $\mu^0(\eta([s,t])) \leq \mu^0(B(0,R))$ on the event $\{ \eta([s,t]) \subset B(0,R) \}$, it follows that for any $0 \leq s \leq t$, $\mu([s,t])$ is a.s. finite.
\end{proof}

\begin{proof}[Proof of Theorem~\ref{thm:mainresult}]
By Lemmas~\ref{lem:conformal_covariance} and~\ref{lem:local_finiteness_eta} and Theorem~\ref{thm:uniqueness_natural_parametrisation} the conclusion of the theorem holds.
\end{proof}

\bibliographystyle{abbrv}
\bibliography{biblio}

\end{document}